\documentclass[12pt]{amsart}

\usepackage{amssymb}

\newtheorem{theorem}{Theorem}[section]
\newtheorem{corollary}[theorem]{Corollary}
\newtheorem{remark}[theorem]{Remark}

\begin{document}
\baselineskip=15pt

\title[Higgs bundles on Gauduchon manifolds]{Stable Higgs 
bundles on compact Gauduchon manifolds}

\author[I. Biswas]{Indranil Biswas}

\address{School of Mathematics, Tata Institute of Fundamental
Research, Homi Bhabha Road, Bombay 400005, India}

\email{indranil@math.tifr.res.in}

\subjclass[2000]{53C07, 32L05}

\keywords{Gauduchon metric, stable Higgs bundle, representation}

\date{}

\begin{abstract}
Let $M$ be a compact complex manifold equipped with a
Gauduchon metric. If $TM$ is holomorphically trivial, and $(V\, 
, \theta)$ is a stable ${\rm SL}(r,{\mathbb C})$--Higgs bundle
on $M$, then we show that $\theta\,=\, 0$. We show that the
correspondence between Higgs bundles and representations of the
fundamental group for a compact K\"ahler manifold does not extend
to compact Gauduchon manifolds. This is done by applying the
above result to $\Gamma\backslash G$, where $\Gamma$ is a discrete
torsionfree cocompact subgroup of a complex semisimple group $G$.

\textsc{R\'esum\'e.} \textbf{Les fibr\'es de Higgs stables sur les
vari\'et\'es de Gauduchon.} Soit $M$ une vari\'et\'e complexe compacte
muni d'une m\'etrique de Gauduchon.
Si $TM$ est holomorphiquement trivial, et $(V, \theta)$ est un fibr\'e
$\text{SL}(r, \mathbb{C})$--Higgs stable, alors on d\'emontre que 
$\theta \,=\,0$.
On d\'emontre que la correspondance entre les fibr\'es de Higgs et les
repr\'esentations du groupe fondamental pour une vari\'et\'e k\"ahlerienne
compacte ne s'\'etend pas aux vari\'et\'es de Gauduchon.  Ceci est accompli
en appliquant le r\'esultat ci-dessus \`a $\Gamma \backslash G$, o\`u $\Gamma$
est un sous-groupe discret, sans torsion et co-compact d'un groupe semi-simple
complexe $G$.
\end{abstract}

\maketitle

\section{Introduction}\label{sec1}

Let $(M\, ,g)$ be a compact connected K\"ahler manifold.
A theorem due to Uhlenbeck and Yau says that the isomorphism classes
of stable vector bundles on $M$ are in bijective correspondence
with the solutions of the Hermitian--Yang--Mills equation on $M$
\cite{UY}. This theorem was later extended to compact complex
manifold equipped with a Gauduchon metric by Li--Yau and
Buchdahl \cite{LY}, \cite{Bu} (the result of \cite{Bu} is
for complex surfaces). Gauduchon metrics are a generalization
of K\"ahler metrics; their definition is recalled in Section
\ref{sec2}.

Hitchin, Simpson, Donaldson and Corlette established 
a bijective correspondence between the isomorphism classes
of stable ${\rm SL}(r,{\mathbb C})$--Higgs bundles on $M$,
with vanishing rational Chern classes, 
and the equivalence classes of irreducible homomorphisms from
$\pi_1(M)$ to ${\rm SL}(r,{\mathbb C})$
\cite{Hi}, \cite{Si}, \cite{Do}, \cite{Co}. It is natural to
ask whether this correspondence extends to compact complex Gauduchon
manifolds. We show that the correspondence does not
extend in general by constructing explicit examples of
compact complex Gauduchon manifolds
for which this correspondence fails.

Let $(M\, ,g)$ be a compact connected complex
Gauduchon manifold. We prove the following theorem:

\begin{theorem}\label{thm0}
If $TM$ is holomorphically trivial, and $(V\, ,
\theta)$ is a stable ${\rm SL}(r,{\mathbb C})$--Higgs bundle on
$M$, then $\theta\,=\, 0$.
\end{theorem}

The above mentioned correspondence valid for K\"ahler manifolds
implies that if $(M\, ,g)$ is a compact connected K\"ahler manifold
such that $\theta\,=\, 0$
for any stable ${\rm SL}(r,{\mathbb C})$--Higgs bundle
$(V\, , \theta)$ on $M$, then any irreducible representation of
$\pi_1(M)$ into ${\rm SL}(r,{\mathbb C})$ is unitarizable
(see Remark \ref{rem1}).

Let $G$ be a connected complex semisimple group
defined over $\mathbb C$. Let
$$
\Gamma\, \subset\, G
$$
be a torsionfree discrete subgroup such that the 
quotient $\Gamma\backslash G$ is compact. This
compact complex manifold $\Gamma\backslash G$ is not
K\"ahler, but it has
explicit Gauduchon metrics. Also, $T(\Gamma\backslash G)$ is
trivial, so Theorem \ref{thm0} applies to it. It turns out that
the restriction of each nontrivial
irreducible representation of $G$ is a nonunitarizable
irreducible representation of $\Gamma$.

\section{Stable Higgs bundles}\label{sec2}

Let $M$ be a compact connected complex manifold of complex
dimension $d$. Let $g$ be a $C^\infty$ Hermitian
structure on the holomorphic tangent bundle $TM$.
Let $\omega_g$ be the positive $(1\, ,1)$--form on $M$ given
by $g$. We recall that $g$ is called a \textit{Gauduchon
metric} if
$$
\partial\overline{\partial} \omega^{d-1}_g\,=\, 0\, .
$$

A theorem due to P. Gauduchon says that given any $C^\infty$ Hermitian 
structure $g_0$ on $TM$, there is a positive smooth function $f$ on $M$
such that $fg_0$ is a Gauduchon metric; furthermore, if $n\, \geq\, 2$,
then $f$ is unique up to a positive constant. (See \cite[p. 502]{Ga}.)

Fix a Gauduchon metric $g$ on $M$. As before, the corresponding
$(1\, ,1)$--form on $M$ will be denoted by $\omega_g$.

Let $F$ be a coherent analytic sheaf on $F$. Consider the determinant
line bundle $\det F$ on $M$; see \cite[Ch. V, \S~6]{Ko} for the
construction of $\det F$. Fix a Hermitian structure $h_F$ on 
$\det F$. Define the \textit{degree} of $F$ to be
$$
\text{degree}(F)\, :=\, \int_M c_1(\det F\, ,h_F)\wedge
\omega^{d-1}_g\, \in\, \mathbb R\, ,
$$
where $c_1(\det F\, ,h_F)$ is the Chern form of the Hermitian
connection on $\det F$. It should be clarified that
$\text{degree}(F)$ is independent of the choice of $h_F$, but
it depends on $g$; see \cite[p. 626]{Bu}, \cite[p. 563]{LY}.

A holomorphic vector bundle $V$ on $M$ is called \textit{stable}
if for every coherent analytic subsheaf $F\, \subset\, V$ with
$0\, <\, \text{rank}(F)\, <\,\text{rank}(V)$, the inequality
$$
\frac{\text{degree}(F)}{\text{rank}(F)} \, 
<\,\frac{\text{degree}(V)}{\text{rank}(V)}
$$
holds.

Let $\Omega^1_M$ be the holomorphic cotangent bundle of $M$.
A \textit{Higgs field} on a holomorphic vector bundle $V$ on $M$
is a section
$$
\theta\, \in\, H^0(M,\, End(V)\otimes\Omega^1_M)
$$
such that $\theta\bigwedge\theta \,=\, 0$.
A Higgs vector bundle is a pair $(V\, ,\theta)$, where $V$ is
a holomorphic vector bundle, and $\theta$ is a Higgs field on $V$;
see \cite{Hi}, \cite{Si}.

A Higgs vector bundle $(V\, ,\theta)$ is called \textit{stable}
if for every coherent analytic subsheaf $F\, \subset\, V$ satisfying
the two conditions that
$\theta(F)\, \subset\, F\otimes \Omega^1_M$ and
$0\, <\, \text{rank}(F)\, <\,\text{rank}(V)$, the inequality
$$
\frac{\text{degree}(F)}{\text{rank}(F)} \,
<\,\frac{\text{degree}(V)}{\text{rank}(V)}
$$
holds.

Let ${\rm trace}\, :\, End(V)\otimes\Omega^1_M\,\longrightarrow\,
\Omega^1_M$ be the homomorphism defined by $\text{trace}\otimes
\text{Id}_{\Omega^1_M}$. So for a Higgs vector bundle $(V\, ,\theta)$,
$$
\text{trace}(\theta)\, \in\, H^0(M,\, \Omega^1_M)\, .
$$

A $\text{SL}(r,{\mathbb C})$--\textit{Higgs bundle} is a
Higgs vector bundle $(V\, ,\theta)$ of rank $r$ such that 
$\det V\, =\, {\mathcal O}_M$ (the trivial line bundle), and 
$\text{trace}(\theta)\,=\, 0$.

\begin{theorem}\label{thm1}
Assume that the tangent bundle $TM$ is holomorphically trivial. Let 
$(V\, , \theta)$ be a stable ${\rm SL}(r,{\mathbb C})$--Higgs bundle
on $M$. Then $\theta\,=\, 0$.
\end{theorem}

\begin{proof}
Fix a holomorphic trivialization of $\Omega^1_M$ by choosing
$d$ linearly independent sections
$$
\beta_i\, \in\, H^0(M,\, \Omega^1_M)
$$
$1\,\leq\, i\,\leq\, d$. Then
$$
\theta\, =\, \sum_{i=1}^d\theta_i\otimes \beta_i\, ,
$$
where $\theta_i\, \in\, H^0(M,\, End(V))$. Since
$\theta\bigwedge\theta \,=\, 0$,
\begin{equation}\label{e1}
\theta_i\circ\theta_j \,=\, \theta_j\circ\theta_i
\end{equation}
for all $i\, ,j \, \in\, [1\, ,d]$.

Assume that $\theta\, \not=\, 0$. Choose $i_0$ such that
$\theta_{i_0}\, \not=\, 0$.

For any point $x\, \in\, M$, let $\lambda_1(x)\, ,\cdots\, 
,\lambda_{n_x}(x)$
be the eigenvalues of $\theta_{i_0}(x)\, \in\, \text{End}_{\mathbb C}
(V_x)$; let $m^x_j$ be the multiplicity of the eigenvalue 
$\lambda_j(x)$. Since all holomorphic functions on $M$
are constants, the characteristic polynomial of $\theta_{i_0}(x)$
is independent of $x$. Hence the collection $\{(\lambda_1(x)\, ,m^x_1)
\, , \cdots\, , (\lambda_{n_x}(x)\, ,m^x_{n_x})\}$
is independent of $x$. Let $V_1\subset\, V$ be the
generalized eigenbundle for the
eigenvalue $\lambda_1(x)$ of $\theta_{i_0}$. So for each point
$x\, \in\, M$, the fiber of $V_1$ over $x$ is generalized eigenspace
of $\theta_{i_0}(x)$ for the eigenvalue $\lambda_1(x)$; it is a
holomorphic subbundle. Let $V^c_1\subset\, V$ be the
holomorphic subbundle given by the direct sum of the generalized 
eigenbundles for all the eigenvalues of $\theta_{i_0}$ different from
$\lambda_1(x)$. Therefore, we have a decomposition
$$
V\,=\, V_1\oplus V^c_1\, .
$$

{}From \eqref{e1} it follows immediately that $\theta_j(V_1)
\, \subset\, V_1$ and $\theta_j(V^c_1)\, \subset\, V^c_1$
for all $j$. Hence $\theta(V_1)\, \subset\,
V_1\otimes \Omega^1$ and $\theta(V^c_1)\, \subset\,V^c_1\otimes 
\Omega^1$. Therefore, if both $V_1$ and $V^c_1$ are nonzero, then
the Higgs bundle $(V\, , \theta)$ decomposes. But a
stable Higgs bundle is indecomposable. Since $(V\, , \theta)$
is stable, we conclude that 
$\theta_{i_0}$ has exactly one eigenvalue. On the other hand,
$\text{trace}(\theta)\,=\, 0$. Hence $0$ is the only
eigenvalue of $\theta_{i_0}$. So, $\theta_{i_0}$ is nilpotent.

Consider the short exact sequence of coherent analytic sheaves on $M$
\begin{equation}\label{e2}
0\, \longrightarrow\, \text{kernel}(\theta_{i_0})
\, \longrightarrow\, V \, \longrightarrow\,
\text{image}(\theta_{i_0})
\, \longrightarrow\, 0\, .
\end{equation}
{}From \eqref{e1} it follows that
$$
\theta_j(\text{kernel}(\theta_{i_0}))\, \subset\,
\text{kernel}(\theta_{i_0}) ~\,~\text{~and~} ~\,~
\theta_j(\text{image}(\theta_{i_0}))\, \subset\,
\text{image}(\theta_{i_0})
$$
for all $j$. Hence
\begin{equation}\label{e3}
\theta(\text{kernel}(\theta_{i_0}))\, \subset\,
\text{kernel}(\theta_{i_0})\otimes\Omega^1_M
~\,~\text{~and~} ~\,~\theta(\text{image}(\theta_{i_0}))\, \subset
\,\text{image}(\theta_{i_0})\otimes\Omega^1_M\, .
\end{equation}

Since $\theta_{i_0}$ is nonzero and nilpotent,
$$
0\, <\, \text{rank}(\text{kernel}(\theta_{i_0}))\, ,
\text{rank}(\text{image}(\theta_{i_0}))\, <\, r\, .
$$
In view of \eqref{e3}, the stability condition for
$(V\, ,\theta)$ says that
\begin{equation}\label{e4}
\frac{\text{degree}(\text{kernel}(\theta_{i_0}))
}{\text{rank}(\text{kernel}(\theta_{i_0}))}\, ,
\frac{\text{degree}(\text{image}(\theta_{i_0}))
}{\text{rank}(\text{image}(\theta_{i_0}))}\, <\,
\frac{\text{degree}(V)}{\text{rank}(V)}\, .
\end{equation}
On the other hand, from \eqref{e2},
$$
\text{degree}(\text{kernel}(\theta_{i_0}))+
\text{degree}(\text{image}(\theta_{i_0}))\,=\,
\text{degree}(V)
$$
and $\text{rank}(\text{kernel}(\theta_{i_0}))+
\text{rank}(\text{image}(\theta_{i_0}))\,=\,
\text{rank}(V)$.
But these contradict \eqref{e4}. Therefore, $\theta\, =\, 0$.
\end{proof}

Theorem \ref{thm1} has the following corollary:

\begin{corollary}\label{cor1}
If $TM$ is holomorphically trivial, and $(V\, ,
\theta)$ is a stable ${\rm SL}(r,{\mathbb C})$--Higgs bundle on
$M$, then the vector bundle $V$ is stable.
\end{corollary}

Let $(M\, ,g)$ be a compact connected K\"ahler manifold.
All the Chern classes will be with rational coefficients.
There is a bijective correspondence between the isomorphism classes
of stable ${\rm SL}(r,{\mathbb C})$--Higgs bundles
$(V\, ,\theta)$ on $M$, with $c_i(V)\,=\, 0$ for all $i\, \geq\, 1$,
and the equivalence classes of irreducible homomorphisms from
$\pi_1(M)$ to ${\rm SL}(r,{\mathbb C})$ (see \cite{Si2} for the
details of this correspondence). Also, there is a bijective 
correspondence between the isomorphism classes of stable vector
bundles $V$ on $M$ of rank $r$ and trivial determinant, with
$c_i(V)\,=\, 0$ for all $i\, \geq\, 1$, and the equivalence classes
of irreducible homomorphisms from $\pi_1(M)$ to ${\rm SU}(r)$
(see \cite{Si2}). The first correspondence is an extension of
the second correspondence: The inclusion of ${\rm SU}(r)$ in
${\rm SL}(r,{\mathbb C})$ gives a map of homomorphisms, and
a stable vector bundles $V$ on $M$ of rank $r$ and trivial determinant
produces a stable ${\rm SL}(r,{\mathbb C})$--Higgs bundle by assigning
the zero Higgs field.

\begin{remark}\label{rem1}
{\rm Assume that the K\"ahler manifold $(M\, ,g)$ has the following
property: If $(V\, , \theta)$ is a stable ${\rm SL}(r,{\mathbb 
C})$--Higgs bundle, then $\theta\,=\, 0$. Take any
irreducible homomorphism
$$
\rho\, :\, \pi_1(M)\,\longrightarrow\,{\rm SL}(r,{\mathbb C})\, .
$$
Let $(V_\rho\, ,\theta_\rho)$ be the stable Higgs bundle corresponding 
to $\rho$. We have $\theta_\rho\,=\,0$ by the assumption on $M$. Hence 
$\rho$ is conjugate to a unitary representation, meaning there is an 
element $A\, \in\, {\rm SL}(r,{\mathbb C})$ such that $A^{-1}(\rho(z))A
\,\in\,{\rm SU}(r)$ for all $z\, \in\, \pi_1(M)$.}
\end{remark}

\section{An example}

Let $G$ be a connected semisimple affine algebraic group
defined over $\mathbb C$; we assume that $G\, \not=\, e$.
Let $d$ be the (complex) dimension of $G$. Let
$$
\Gamma\, \subset\, G
$$
be a torsionfree discrete subgroup such that $\Gamma\backslash G$
is compact. Since $\Gamma\backslash G$ is compact, the subgroup
$\Gamma$ is Zariski dense in $G$ \cite{Bo}. (See \cite{Ra} for
such manifolds.)

We note that there are explicit Gauduchon metrics on
the complex manifold $\Gamma\backslash G$. Indeed,
take any Hermitian metric $h$ on $\Gamma\backslash G$ given by
some left translation invariant Hermitian metric $\widetilde{h}$
on $G$. Let $\omega_h$ and $\omega_{\widetilde{h}}$ be the corresponding
$(1\, ,1)$--forms on $\Gamma\backslash G$ and $G$ respectively.
Since $\widetilde{h}$ is left translation invariant, the top degree
form $\partial\overline{\partial} \omega^{d-1}_{\widetilde{h}}$ is
also left translation invariant. Hence $\partial\overline{\partial} 
\omega^{d-1}_{\widetilde{h}}\,=\, c_0\cdot \mu_G$, where $c_0\, \in\, 
\mathbb R$,
and $\mu_G$ is the Haar measure form on $G$. The form
$\partial\overline{\partial} \omega^{d-1}_h$ is closed because
$\partial\overline{\partial} \omega^{d-1}_h\,=\,
d\overline{\partial} \omega^{d-1}_h$. Hence using Stokes' theorem,
$$
0\,=\, \int_{\Gamma\backslash G} \partial\overline{\partial}
\omega^{d-1}_h \,=\, \int_{\Gamma\backslash G} c_0\mu_G
\,=\, c_0\text{Vol}_{\mu_G}(\Gamma\backslash G)\, .
$$
Therefore, $c_0\,=\, 0$. Hence $h$ is a Gauduchon metric.

The holomorphic tangent bundle $T(\Gamma\backslash G)$ is trivial
(a trivialization is given by any left translation invariant 
trivialization of $TG$).

It can be shown that $\Gamma\backslash G$ does not admit any
K\"ahler metric. Indeed, any compact connected K\"ahler manifold 
with trivial tangent bundle is isomorphic to a complex torus,
implying that its fundamental group is abelian. But the
fundamental group $\Gamma$ of $\Gamma\backslash G$ is not abelian.
(Since $\Gamma$ is Zariski dense in $G$, if $\Gamma$ is
abelian, then $G$ is abelian.)

Take any nontrivial irreducible representation
$$
\rho'\, :\, G\, \longrightarrow\, \text{SL}(V_0)\, .
$$
Let
\begin{equation}\label{e5}
\rho\, :=\, \rho'\vert_\Gamma
\end{equation}
be the restriction of $\rho'$ to the subgroup $\Gamma$.

We have $\pi_1(\Gamma\backslash G)\,=\, \Gamma$. Since 
$\Gamma$ is Zariski dense in $G$, the restriction
$\rho$ in \eqref{e5} remains irreducible. Since
$$
\rho(\Gamma)\, \subset\, \text{SL}(V_0)
$$
is an infinite, closed and discrete subgroup, it cannot be
conjugated, by some element of $\text{SL}(V_0)$, to a
subgroup of a maximal compact subgroup of $\text{SL}(V_0)$
(every closed infinite subgroup of a compact group has a limit 
point, hence it is not discrete).
Compare this with Remark \ref{rem1}.


\end{document}